\documentclass[12pt, a4paper]{amsart}
\usepackage{amssymb,amsmath,mathrsfs}
\usepackage[colorlinks=true,urlcolor=blue,
citecolor=red,linkcolor=blue,linktocpage,pdfpagelabels,bookmarksnumbered,bookmarksopen]{hyperref}
\usepackage[english]{babel}

\newcommand{\N}{{\mathbb N}}

\newcommand{\R}{{\mathbb R}}

\newcommand{\rn}{{\mathbb{R}^N}}
\newcommand{\be}{\begin{equation}}
\newcommand{\ee}{\end{equation}}

\newcommand{\ov}[1]{\overline{#1}}

\numberwithin{equation}{section}
\newtheorem{theorem}{Theorem}
\newtheorem{proposition}[theorem]{Proposition}

\newtheorem{lemma}[theorem]{Lemma}

\theoremstyle{definition}
\newtheorem{remark}{Remark}

\newtheorem{open}{Open Problem}

\newcommand{\brm}{\begin{remark}\rm}
\newcommand{\erm}{\end{remark}}
\newcommand{\brms}{\begin{remark}\rm}
\newcommand{\erms}{\end{remark}}
\newcommand{\bte}{\begin{theorem}}
\newcommand{\ete}{\end{theorem}}
\newcommand{\bpr}{\begin{proposition}}
\newcommand{\epr}{\end{proposition}}
\newcommand{\ble}{\begin{lemma}}
\newcommand{\ele}{\end{lemma}}
\newcommand{\beq}{\begin{equation}}
\newcommand{\eeq}{\end{equation}}
\newcommand{\bdm}{\begin{displaymath}}
\newcommand{\edm}{\end{displaymath}}
\numberwithin{equation}{section}

\newcommand{\bos}{\begin{remark}\rm}
\newcommand{\eos}{\end{remark}}

\newcommand{\ben}{\begin{enumerate}}
\newcommand{\een}{\end{enumerate}}

\title[A problem with quadratic growth in the gradient]
{Existence and multiplicity for elliptic problems with quadratic growth in the gradient}

\author[Louis Jeanjean]{Louis Jeanjean}
\address{Louis Jeanjean
\newline\indent
Laboratoire de Math\'ematiques (UMR 6623)
\newline\indent
Universit\'{e} de Franche-Comt\'{e}
\newline\indent
16, Route de Gray 25030 Besan\c{c}on Cedex, France}
\email{louis.jeanjean@univ-fcomte.fr}

\author[Boyan Sirakov]{Boyan Sirakov}
\address{Boyan Sirakov
\newline\indent
 Pontif\'{\i}cia Universidade Cat\'olica do Rio de Janeiro
 \newline\indent
Departamento de Matem\'atica,
\newline\indent
Rua Marqu\^es de S\~ao Vicente 225, G\'avea
\newline\indent
Rio de Janeiro - RJ,
CEP 22451-900,
\newline\indent Brazil }
\email{bsirakov@mat.puc-rio.br}

\begin{document}


\begin{abstract}
We show that a class of divergence-form elliptic problems with quadratic growth in the gradient and non-coercive zero order terms are solvable, under essentially optimal hypotheses on the coefficients in the equation. In addition, we prove that the solutions are in general not unique. The case where the zero order term has the opposite sign was already
intensively studied and the uniqueness is the rule.
\end{abstract}
\maketitle



\section{Introduction}

Boundary value problems for elliptic equations like
\begin{equation}\label{0.2}
-\mathrm{div}(a(x,u, \nabla u)) = B(x,u,\nabla u) + f(x), \quad x \in \Omega\subset\rn,
\end{equation}
where $-div(a(x,\cdot,\nabla \cdot))$ is a Leray-Lions operator on
some Sobolev space, have been one of the central problems in the
theory of elliptic PDE in divergence form. This paper is a
contribution to this study for the widely explored case when the
nonlinear term  $B(x,u,\xi)$ has ``natural growth" in the unknown
function, that is, grows linearly in $u$ and quadratically in
$\xi\in \rn$. The model case for our study is
\begin{equation}\label{model}
a(x,u, \xi) = A(x)\xi, \qquad B(x,u,\xi)  = c_0(x) u + \mu(x)|\xi|^2,
\end{equation}
where $A$ is a positive bounded matrix, $\mu\in L^\infty(\Omega)$,
and $c_0, f$ belong to suitably chosen Lebesgue spaces.

This type of problems have generated a considerable literature.
 Let us mention here~\cite{BoMuPu2, BoMuPu3, DaGiPu, FeMu2, GrMuPo, BaBlGeKo,BaMu,BaPo, AbDaPe, AbBi, AbBi2} as reference
papers on this subject, most closely related to the problem we
consider. In these works
the existence, uniqueness or multiplicity of solutions of
\eqref{0.2} is established under various conditions on $a$,  $B$
and $f$, which will be discussed below.

Most of the works quoted above, when reduced to \eqref{model}, assume that the coefficient $c_0$ is nonpositive, that is, the equation is {\it coercive} or {\it proper}. The only exception to this rule is \cite{AbBi2}, in which the particular case $c_0=f\gneqq0$ in the model problem \eqref{0.2}-\eqref{model} was mentioned; in the next section we will give a more detailed account on the results which appeared prior to this paper. Here we consider the general problem \eqref{Generalf} below, with
non-coercive dependence in the unknown function $u$. Specifically, we are going to see that, {\it when $c_0$ is
positive and sufficiently close to zero,  the same type of
existence result as in the case $c_0\le0$ can be obtained, but the
bounded solutions are not unique}.

The paper is organized as follows. The next section contains our hypotheses and main results,  and situates them with respect to previous works. A brief overview of the proofs is given in Section \ref{diff}, while the proofs themselves can be found in Sections \ref{link}--\ref{nonconstant}. We conclude with some final remarks in Section \ref{final}, where we discuss possible extensions and open problems.

\section{Main Results}

In this section we state our main results. We study the equation
\begin{equation}\label{Generalf}
 - \mathrm{div}(A(x) \nabla u) =  H(x,u,\nabla u), \quad u \in H^1_0(\Omega),
\end{equation}
 where    $\Omega \subset \R^N, N \geq 3$ is a bounded domain in
$\R^N$,
$$
\mathbf{(H1)}\;\left\{ \begin{array}{c} A \in L^{\infty}(\Omega)^{N \times N},\; \Lambda I\ge A \geq \lambda I,\;\mbox{ for }\mbox{ some } \Lambda \ge \lambda>0,\mbox{ and }\\ \\
    |H(x,s,\xi)|\le c_0(x)|s|+ \mu|\xi|^2 + f(x),\\ \\
\mbox{ for }\mbox{ some }\mu\in \mathbb{R}^+\,,\; c_0, f\in L^p(\Omega)\mbox{
with } p
>\frac{N}{2}. \end{array}
\right.
$$
In the sequel we denote with $C_N$ the optimal Sobolev constant, defined in \eqref{sobo} below.
We have the following main existence result.
\begin{theorem} \label{theoexist}
\ \ Assume that (H1)  holds and
\begin{equation}\label{H1}
 \mu\|f\|_{L^{\frac{N}{2}}(\Omega)} <{C_N}.
\end{equation}
 Then there
exists a constant $\ov{c}>0$ depending on $N$, $p$,  $|\Omega|$,
$\mu$, $\mu\|f\|_{L^p(\Omega)}$, such that if
$$ \|c_0\|_{L^p(\Omega)}< \ov{c}$$
then \eqref{Generalf} admits a bounded solution.
\end{theorem}

Next, we  show that introducing a non-coercive zero
order term in \eqref{Generalf} induces {\it non-uniqueness} of the
bounded solutions of this equation, in the extremal cases of
the structural hypothesis (H1) above. In other words, we prove a
multiplicity result for the equation
\begin{equation} \label{problem0}
 -\Delta u = c_0(x)u + \mu |\nabla u|^2 + f(x), \qquad  u \in
H^1_0(\Omega),
\end{equation}
where $\mu\in \mathbb{R}$, $c_0, f\in L^p(\Omega)$.
\begin{theorem} \label{theorem1} Assume that $$\mu\not=0\quad\mbox{  and }\quad c_0\gneqq0\;\mbox{ in } \Omega.$$
If
\begin{equation}\label{H2}\|[\mu f]^+\|_{L^{\frac{N}{2}}(\Omega)}
<{C_N},
\end{equation}
and
$$
\max\{\|c_0\|_{L^p(\Omega)}\,,\,\|[\mu f]^-\|_{L^p(\Omega)}\}<
\ov{c}\,,
$$
where $\ov{c}>0$ depends only  on $N$, $p$, $|\Omega|$, $|\mu|$,
$\|[\mu f]^+\|_{L^p(\Omega)}$, then \eqref{problem0} admits at
least  $\mathrm{two}$ $\mathrm{bounded}$ solutions.
\end{theorem}

\begin{remark} It is easy to check that the hypotheses in the above theorems are necessary, in the sense that \eqref{problem0} has no bounded solutions if $c_0=0$ and $\mu$ is large, or if $c_0=0$ and $f$ is large, or if $\mu = 0$ and $c_0$ is large; also if $c_0=0$ or $\mu=0$ the solution given by Theorem \ref{theoexist} is unique. See for instance the last remarks in Section 3 of \cite{Si}, pages 598-599 in that paper.
\end{remark}
\begin{remark} Note that in Theorem \ref{theorem1} there is no restriction on the sign of the source term $f(x)$.
\end{remark}
\begin{remark} A slightly more general version of Theorem \ref{theorem1} will be given in Section \ref{nonconstant} (see also the remarks in Section \ref{final}).
\end{remark}

Next we review the existence and uniqueness
results which appeared prior to our work. Because of the very
large literature we restrict ourselves to  works which encompass the model
case \eqref{model} (the reader may consult the references in the
papers quoted below for various related problems).

We begin with references concerning Theorem \ref{theoexist}. A weaker version
of this result appeared already in Kazdan and Kramer \cite{KaKr}, where
equations in non-divergence form are studied. Later, in \cite{BoMuPu2, BoMuPu3},
Boccardo, Murat and Puel showed that the sub- and super-solution method applies
to general divergence-form equations with quadratic growth in the gradient, and
proved existence of bounded solutions of such equations under a hypothesis of
strict coercivity in $u$, that is, $c_0(x)\le -\alpha_0<0$ in \eqref{model}.
For results on strictly coercive equations we refer also to dall'Aglio,
Giachetti and Puel \cite{DaGiPu}. Next, the equation \eqref{Generalf} in the weakly coercive case  ($c_0\equiv 0$) was studied by Maderna, Pagani and Salsa \cite{MaPaSa}, and
Ferone and Murat \cite{FeMu1, FeMu2}.  In that case it turns out that existence
can be proved only under a smallness hypothesis on the source term $f$, as in
\eqref{H1}. Theorem \ref{theoexist} reduces to these results when $c_0=0$, and
extends them to non-coercive zero-order terms.

Let us give some more context on coercive problems. Further existence results with
weaker assumptions of regularity on the coefficients can be found in Grenon,
Murat and Porretta \cite{GrMuPo}. Uniqueness results in
natural spaces associated to the coercive problem were obtained by Barles and Murat
\cite{BaMu}, Barles, Blanc, Georgelin, and Kobylanski \cite{BaBlGeKo}, Barles
and Porretta \cite{BaPo}. We also refer to the recent works by Abdellaoui,
dall'Aglio and Peral \cite{AbDaPe}, and Abdel Hamid and Bidaut-V\'eron
\cite{AbBi} for a deep study of \eqref{problem0} in the particular case $c_0= 0$, $\mu=1$, and
$f\ge 0$. They show that in this case the problem
\eqref{problem0} has infinitely many solutions, of which only one is such that
$e^u-1\in H^1_0(\Omega)$. For results on other classes of equations of type
\eqref{0.2}, with $H$ being for instance in the form $H(x,s,\xi)=
\beta(s)|\xi|^2$ for some real function $\beta$, we refer to Boccardo,
Gallou\"et, and Murat \cite{BoGaMu},  as well as to \cite{MaPaSa},
\cite{AbDaPe}, \cite{AbBi,AbBi2}. We note that in many of these
papers equations involving quasilinear operators modeled on the $p$-Laplacian
are also studied. Finally, the second author \cite{Si} recently obtained
existence and uniqueness  results for fully nonlinear equations in
non-divergence form with quadratic dependence in the gradient, in which case
the adapted weak notion of solution is the viscosity one (see \cite{Si} for
references on these types of problems).
The idea of
our study originated from that paper.

As far as Theorem \ref{theorem1} is concerned,
the fact that in problems with natural growth in the gradient the
presence of a non-coercive zero-order term may lead to
non-uniqueness of bounded solutions was observed only very
recently in \cite{Si}, for the equation \eqref{problem0} with
$f=0$. Subsequently the case when $f\equiv c_0\gneqq0$ was
considered in the work by Abdel Hamid and Bidaut-V\'eron
\cite{AbBi2} (their model equation is $-\Delta_p u = |\nabla u|^p
+ \lambda f(x)(1+u)^b$, $b\ge p-1$). Theorem \ref{theorem1} is valid for arbitrary source term $f$, which in particular may not be positive, and thus shows the multiplicity result is independent of the source term
-- as long as it has a small norm, of course, otherwise solutions
may not exist.

To summarize, Theorem \ref{theoexist} is an essentially optimal, with respect to the coefficients, result on existence of bounded solutions of \eqref{Generalf},  for equations in divergence form with possibly non-coercive zero-order terms; while Theorem \ref{theorem1} shows uniqueness of bounded solutions is lost in the presence of non-coercive zero-order terms, at least  in the model cases. We do not know whether a more general non-uniqueness result is valid (see Section \ref{final}).

In the next section we give more details on the underlying ideas in our approach, and discuss the difference between coercive and non-coercive problems.
\vspace{1cm}
\bigskip

\begin{center}\textbf{Notation.}\end{center}
\begin{enumerate}{\small
\item We denote by $X$ the space $H^1_0(\Omega)$ equipped with the Poincar\'{e}
norm $|||u|||:=\int_\Omega |\nabla u|^2$, and by $X^{-1}$ its dual.
\item For $v \in L^1(\Omega)$ we define $v^+= max(v,0)$ and $v^- =
max(-v,0)$.
\item  The norm $(\int_{\Omega}|u|^pdx)^{1/p}$
in $L^p(\Omega)$ is denoted by $\|\cdot\|_p$. We denote by $p^{\prime}$
the conjugate exponent of $p$, namely $p^{\prime} = (p-1)/p.$
\item We denote by $C,D>0$ any positive constants which are not
essential in the arguments and may vary from one line to another.}
\end{enumerate}
\section{Discussion and general frame of the proofs}\label{diff}

The aim of this section is, first, to provide
some intuition on the hypotheses in our theorems and the differences they introduce with respect to previous works on problems with natural growth in the gradient, and second, to describe the ideas of the proofs of Theorem \ref{theoexist} and Theorem  \ref{theorem1}.
To this goal, and in order to help the reader understand why the case $c_0^+ \not\equiv  0$ is different from the
cases $c_0 \leq 0$ or $c_0 \leq - \alpha_0 <0$, we present a variational interpretation of the model problem \eqref{problem0}.

Let us assume, for the time being, that $\mu >0$ is a constant and $c_0$
and $f$ are smooth functions. Making the
well-known
 change of unknown $ v= \frac{1}{\mu}(e^{\mu
u}-1)$ in \eqref{problem0}  we observe that if a solution
of
\begin{equation}\label{0.3}
- \Delta v - [c_0(x) + \mu f(x)] v = c_0(x) g(v) + f(x), \quad v \in
X,
\end{equation}
where
\begin{equation}\label{defg}
g(s) = \begin{cases} \text{$\frac{1}{\mu} (1+ \mu s) ln (1+ \mu s) - s \quad $   if $ \quad s > - \frac{1}{\mu}$} & \\
\noalign{\vskip4pt} \text{ $-s \quad$ if $\quad s \leq -
\frac{1}{\mu}\,,$} &
\end{cases}
\end{equation}
satisfies $v > - \frac{1}{\mu}$, then $u =
\frac{1}{\mu}ln(1+ \mu v)$ is a solution of~\eqref{problem0}. In the next section we are going to see that this procedure can be made rigorous in general, and we will obtain a priori bounds on solutions of \eqref{0.3} which show that they indeed give solutions of \eqref{problem0}, under the hypotheses of our theorems.

Equation~\eqref{0.3} admits a variational formulation, in other words, its solutions in $H^1_0(\Omega)$ can be represented as critical points of a functional defined on this space. Specifically,
critical points of
$$
       I(v)   := \frac{1}{2}\int_{\Omega}|\nabla v|^2 -[c_0(x) + \mu
f(x)] v^2 dx   - \int_{\Omega}c_0(x)G(v) \thinspace
dx   - \int_{\Omega}f(x)v \thinspace dx
$$
 on $H_0^1(\Omega)$ are weak solutions of~\eqref{0.3}. Here $G(s) = \int_0^s g(t)
\thinspace dt$.

No such link between problems of type \eqref{Generalf} and problems admitting a variational formulation has appeared in the earlier works on coercive equations with natural growth in the gradient \cite{BoMuPu2, BoMuPu3, DaGiPu, MaPaSa, FeMu1, FeMu2}.
The fact that the general problem \eqref{Generalf} does not have such a formulation surely explains this; however the validity of the results obtained in these papers can be explained, in a different light, by looking at the model problem \eqref{0.3}.

First, if we assume that $$c_0 \leq -\alpha_0 <0$$ (as in \cite{BoMuPu2, BoMuPu3, DaGiPu}), it is easily seen that we have, independently of the size of $f\in L^{\frac{N}{2}}(\Omega)$,
\begin{equation}\label{coer}
\lim_{|||v|||\to\infty} I(v) = + \infty,
\end{equation}
or in other words $I$ is {\it coercive}, from which the
existence of a global minimum of $I$ follows. Indeed, to prove \eqref{coer} we observe that the second term in the definition of $I$ dominates, for $|||v|||$ large, the first and third terms, since (see Lemma \ref{prop-g1})
\begin{equation}\label{infy}
\lim_{s\to \infty}\frac{G(s)}{s^2} = + \infty.
\end{equation}

Next, if $c_0 =0$ (as in \cite{MaPaSa}, \cite{FeMu1}, \cite{FeMu2}), then $I$
becomes
$$I(v) = \frac{1}{2}\int_{\Omega}|\nabla v|^2 - \mu f(x)v^2 \thinspace dx -
\int_{\Omega}f(x) v \thinspace dx\,$$ and it is easily seen that this functional is
coercive if and only if (see Lemma \ref{positivity})
\begin{equation}\label{0.4}
\inf_{\|v\|_{L^2(\Omega)} =1} \int_{\Omega}|\nabla v|^2 - \mu f(x)
v^2 dx
>0,
\end{equation}
which in turn holds under the condition \eqref{H1},  discovered in~\cite{FeMu1}.
\medskip

On the other hand, in the case we are interested in
 $$
 c_0^+\not\equiv0,$$
 the geometry of $I$ is completely different, now \eqref{infy} implies
\begin{equation}\label{nocoer}\inf_{v\in X} I(v) =\liminf_{|||v|||\to\infty} I(v)= -\infty,\end{equation} and in particular no global minimum of $I$ exists. \medskip

However, as we are going to see, it turns out that if $c_0^+$ is appropriately small, the functional $I$ takes strictly positive values on the boundary of some {\it large} ball $B$ in $H^1_0(\Omega)$.
In other words, we show that letting the coefficient $c_0$  be slightly positive perturbs badly $I$ at infinity (compare \eqref{coer} to \eqref{nocoer}) but keeps $I$ ``sufficiently large" on some large sphere. Hence, in view of
$I(0)=0$, it follows that  $I$ attains a {\it local} minimum in $B$, which is then a critical point of $I$.

The latter argument applies only to the extremal case  \eqref{problem0} but yields existence for  general equations as in  Theorem~\ref{theoexist}, via the method of sub- and super-solutions.  This method requires no variational structure at all, and applies to very general equations (see for instance \cite{AmCr, BoMuPu4, DeHe}). Note that the method of sub- and super-solutions is particularly useful in searching for stable solutions, and a local minimum of a functional corresponds precisely to a stable solution.

Let us now explain why Theorem \ref{theorem1} is valid. The existence of a local minimum of $I$ and \eqref{nocoer} suggest that  at least one more
critical point (of saddle type) of $I$ could be expected to exist. Proving this type of statement is the object of a large branch of the theory of variational methods in PDE, whose development started with the acclaimed work by Ambrosetti and Rabinowitz \cite{AmRa} on functionals which have ``mountain-pass" geometry, that is, are positive on a small sphere and tend to $-\infty$ at infinity. In our  case we are able to prove that a second critical point of $I$ exists by showing that Cerami sequences for $I$ are bounded, from which classical arguments permit us to deduce the result. The boundedness of Cerami sequences is a significant difficulty and to overcome it we need to develop further some ideas introduced in \cite{Je2}.

It is important to note that the latter argument depends strongly on the variational structure of the PDE in consideration, which is the reason for which we are able to prove Theorem \ref{theorem1} only for  the model equation \eqref{problem0}.  See Open Problem 1 in Section \ref{final}, and the remarks therein. \medskip

Here is an outline of  the following sections.  First, in Section \ref{link} we give some preliminaries and study the relation between the problems \eqref{problem0} and \eqref{0.3}. In Section~\ref{geometry} we establish several facts on the geometry of the functional $I(v)$, and show it admits a local minimum. The core of the multiplicity result is in Section \ref{cerami}, where we show that Cerami sequences for I are bounded. In Section~\ref{nonconstant} we finish the proof of  Theorems~\ref{theoexist} and~\ref{theorem1}. Section \ref{final} contains  some closing remarks and open problems.

\section{The link between problems \eqref{problem0} and \eqref{0.3}}\label{link}

We consider the problem
\begin{equation}\label{2.11}
- \Delta v - [c_0(x) + \mu f(x)] v = c_0(x) g(v) + f(x), \quad v
\in X,
\end{equation}
where $g$ is given by \eqref{defg} and $\mu>0$.

\begin{lemma}\label{dual1}
If $v \in X$
is a solution of~\eqref{2.11} which satisfies $$ v
> - 1 /\mu+\varepsilon\quad\mbox{on }\;\Omega,\qquad\mbox{for some }\;\varepsilon>0,$$ then $ u = \frac{1}{\mu} ln (1 + \mu v)$ is a solution of~\eqref{problem0}.
\end{lemma}

\begin{proof}
The equation~\eqref{2.11} can be rewritten, for $ v
> - 1 /\mu$,
\begin{equation}\label{2.12}
- \Delta v = \frac{c_0(x)}{\mu}(1 + \mu v) ln (1 + \mu v) + (1 + \mu v) f(x).
\end{equation}
Let $v \in X$ be a  solution of~\eqref{2.12}, we want to show that
$ u = \frac{1}{\mu}ln(1+ \mu v)$ is a  solution
of~\eqref{problem0}, that is, if $\phi \in
C_0^{\infty}(\Omega)$, then
\begin{equation}\label{2.13}
\int_{\Omega} \nabla u \nabla \phi - \mu |\nabla u|^2 \phi -
c_0(x) u \phi \thinspace dx = \int_{\Omega} f(x) \phi \thinspace
dx.
\end{equation}
Let $\displaystyle \psi = \frac{\phi}{1+ \mu v}.$ Clearly $\psi
\in X$ and thus it can be used to test~\eqref{2.12}. We get
\begin{equation}\label{2.15}
\int_{\Omega} \nabla v \nabla \psi \thinspace dx = \int_{\Omega}
\frac{c_0(x)}{\mu} ln(1+ \mu v) \phi \thinspace dx +
\int_{\Omega}f(x) \phi \thinspace dx.
\end{equation}
But
\begin{equation}\label{2.16}
\int_{\Omega} \frac{c_0(x)}{\mu} ln(1+ \mu v) \phi \thinspace dx =
\int_{\Omega}c_0(x) u \phi \thinspace dx
\end{equation}
and
\begin{eqnarray}
      \nonumber \int_{\Omega} \nabla v \nabla \psi \thinspace dx  & = &
       \int_{\Omega}\nabla \left( \frac{1}{\mu}(e^{\mu u}-1)\right)\nabla \left(\frac{\phi}{1+ \mu v}\right)dx \\
       \nonumber
& = &     \int_{\Omega} e^{\mu u} \nabla u \left( \frac{\nabla \phi}{1+ \mu v}-
\frac{\mu \phi \nabla v }{(1+ \mu v)^2}\right) dx \\
\nonumber & = & \int_{\Omega} \nabla u \left( \nabla \phi -
\frac{\mu \phi \nabla
(\frac{1}{\mu}(e^{\mu u}-1))}{(1+ \mu v)} \right)dx \\
 & =&  \int_{\Omega }\nabla u (\nabla \phi - \mu \phi
\nabla u ) \thinspace dx \nonumber \\
\label{2.17} & =&  \int_{\Omega} \nabla u \nabla \phi -
\mu |\nabla u|^2 \phi \thinspace dx.
   \end{eqnarray}
Combining~\eqref{2.15},~\eqref{2.16} and~\eqref{2.17}, we see that $u$
satisfies~\eqref{2.13}.
\end{proof}

Next we recall the following standard fact.

\begin{lemma}\label{positivity}
 Given $h\in L^{N/2}(\Omega)$, set
$$
E_h^2(u) =  \int_{\Omega} |\nabla u|^2 - h(x) |u|^2 dx,
$$
for $u\in X$. Then
  $$\|h^+\|_{\frac{N}{2}} < C_N$$  implies that the quantity $E_h(u)$ defines a norm on $X$
which is equivalent to the standard norm, and $$\lambda(h,\Omega) := \inf_{u\in X\setminus\{0\}}\frac{ E_h (u)}{|||u|||^2}>0.$$
This last property implies that the operator $-\Delta - h $ satisfies the maximum principle in $\Omega$, that is, if $-\Delta u - hu \ge 0$ in $X^{-1}$ for some $u\in X$, then  $u^-\in X$ yields $u^-\equiv0$ in $\Omega$.
\end{lemma}

\begin{proof}
 The first statement trivially follows from the Sobolev embedding and the fact that for any $v \in X,$
\begin{equation}\label{2.18}
\int_{\Omega} h(x) v^2 dx \leq  \|h\|_{\frac{N}{2}} \|v\|_{2^*}^2
\leq  \frac{1}{C_N} \|h\|_{\frac{N}{2}} \|\nabla v\|_2^2.
\end{equation}
Here $2^* = \frac{2N}{N-2}$ and
\begin{equation}\label{sobo}
C_N = inf \{ \|\nabla v\|_2^2 \;:\; v\in X,\;\|v\|_{2^*}^2=1\}>0
\end{equation}
is the optimal constant in Sobolev's inequality.
Note $C_N$ depends only on $N$; the exact value of $C_N$ can be found  in \cite{Ta}. The maximum principle is obtained by multiplying $-\Delta u - hu \ge 0$ by $u^-$ and by integrating.
\end{proof}\smallskip

\noindent {\bf Definition}. {\it In the rest of the paper we assume that the constant $\ov{c}>0$ in the main theorems is fixed so small that
\medskip

\noindent ${\bf (H2)}$ \centerline{$||c_0 + \mu f^+||_{\frac{N}{2}} < C_N.$}\\

We denote with $||\cdot||$ the norm  defined by $E_{c_0+\mu f}(\cdot)$ which, by Lemma~\ref{positivity}, is equivalent to the standard norm on $X$.}

By \eqref{H2}, the validity of  (H2) can be ensured by taking
\begin{equation}\label{smallc}
\|c_0\|_{\frac{N}{2}}\le \varepsilon_0/2,\qquad\mbox{with}\qquad
\varepsilon_0:=C_N-\mu\|f^+\|_{\frac{N}{2}}>0,
\end{equation}
 which occurs for $||c_0||_p$ sufficiently
small, since $|\Omega| < \infty$.

 Now we recall the following global boundedness lemma, which is a
consequence of results due to Stampacchia and Trudinger.
\begin{lemma}\label{trudi} Assume that $A \in L^{\infty}(\Omega)^{N \times N}$, $\Lambda I\ge A \geq \lambda I$, for some  $\Lambda \ge \lambda>0$, and that $c,f \in L^p(\Omega)$ for some $p>\frac{N}{2}$. Then if $u\in X$ is a solution of
$$
-\mathrm{div}(A(x)\nabla u) \le (\ge) c(x) u + f(x)
$$
then $u$ is bounded above(below) and
$$
\sup_{\Omega} u^+ (\sup_\Omega u^-) \le C(\|u^+(u^-)\|_2 + \|f\|_p),
$$
where $C$ depends on $N, p, \lambda,\Lambda,|\Omega|$, and $\|c\|_p$.
\end{lemma}
\begin{proof} This is a consequence of Theorem 4.1 in \cite{Tr} combined with Remark 1 on page 289 in that paper. It can also be obtained by repeating the proof of Theorem 8.15 in \cite{GT} (which implies the same result for $c\in L^\infty(\Omega)$), as remarked at the end of page 193 in that book.
\end{proof}

The next lemma shows that Lemma \ref{dual1} can be applied,
provided the function $c_0$ is sufficiently small.

\begin{lemma}\label{negative-bound}
 There exists a constant $\ov{c}>0$ depending on $N, p$, $|\Omega|$,
  $\mu\|f^{+}\|_p$, such that if
$$
\mu\|f^{+}\|_{\frac{N}{2}}
<{C_N},\qquad\mbox{and}\qquad
\max\{\|c_0\|_{p}\,,\,\mu\|f^{-}\|_{p}\}<
\ov{c},
$$ then any solution $v$
of~\eqref{2.11} satisfies $v > - 1 / (2\mu) $ in $\Omega$.
\end{lemma}

\begin{proof} Since $c_0 \geq 0$ on $\Omega$ and
$g$ is nonnegative on $\R$,  any solution of~\eqref{2.11}
satisfies
\begin{equation}\label{2.112}
- \Delta v - [c_0(x) + \mu f^+(x)] v \geq - f^-\quad\mbox{ on }\mbox{ the }\mbox{ set } \{v<0\}.
\end{equation}
We now use the global bound given in the previous lemma to infer that
\begin{equation}\label{2.113}
\sup_\Omega(v^-) \leq C (||v^-||_2 + ||f^-||_p),
\end{equation}
for some constant $C = C(N,p, |\Omega|, ||c_0+\mu f^+||_p)$.

Recall that if we assume $\ov{c}>0$ is small enough (H2) holds, and thus $||\cdot||$ is equivalent to the standard norm on $X$. We multiply~\eqref{2.11} by $v^-$,  and integrate to get
\begin{eqnarray*}
||v^-||^2  & \leq & \int_{\Omega}|\nabla v^-|^2 - [c_0(x) + \mu f(x)]
|v^-|^2 dx\\
&\leq& - \int_{\Omega}c_0(x) g(v) v^- dx
- \int_{\Omega}f(x) v^- dx\\
&\leq&  \int_{\Omega}f^-(x) v^- dx \\
&\leq&   ||f^-||_{\frac{N}{2}} ||v^-||_{\frac{N}{N-2}} \leq C
||f^-||_{\frac{N}{2}}||v^-||.
\end{eqnarray*}
Thus, in particular,
\begin{equation} \label{2.115}
||v^-||_2 \leq  C ||f^-||_p.
\end{equation}
Combining~\eqref{2.113} and~\eqref{2.115} we get $\sup(v^-) \leq C
||f^-||_p$. This implies that  $\sup(v^-) < 1 /(2\mu)$ if $C ||f^-||_p < 1 /(2\mu) $, that is,  if $||\mu f^-||_p \leq 1/ (2C) $. This finishes the proof.
\end{proof}

\section{On the geometry of the functional $I(v)$}\label{geometry}

 We associate to
\eqref{2.11} the functional $I : X \to \R$ defined by
$$ I(v) = \frac{1}{2}||v||^2- \int_{\Omega}c_0(x) G(v) \thinspace dx -
\int_{\Omega}f(x) v \thinspace dx.$$ Under our assumptions it is
standard to show that $I\in C^1(X, \R)$\footnote{Note that in
Sections \ref{geometry} and \ref{cerami} $\mu$ can be an arbitrary
 function in $L^\infty(\Omega)$}. \vskip2pt \noindent

Recall $G(s) = \int_0^s g(t)dt$ and define $H(s) = \frac{1}{2}g(s)s
-G(s).$ In the following lemma we gather some simple and useful properties of
$g, G$ and $H$.

\begin{lemma}\label{prop-g1}
\vskip2pt \noindent
\begin{itemize}
\item[(i)] The function $g$ is continuous on $\R$, $g > 0$ on $\R\setminus\{0\}$, $G \geq 0$ on~$\R^+$
and $G \leq 0$ on $\R^-$. \item[(ii)] For any $r\in(1,2)$ there exists $C=C(r,\mu) >0$  such that we have $|g(s)| \leq C |s|^r$ for any $s\in \mathbb{R}$.
\item[(iii)] We have  $ g(s)/s \to 0$ as $s \to 0$.
\item[(iv)] We have $g(s)/s \to + \infty$ and $G(s)/s^2 \to + \infty$ as $s \to + \infty$.
\item[(v)] The function $H$ satisfies
$\displaystyle H(s) \leq (s/t) H(t)$, for $0 \leq s \leq t$.
\item[(vi)] The function $H$ is bounded on $\R^-$.
\end{itemize}
\end{lemma}

\begin{proof}
We have $g(0)= 0$ and, for $s
> - 1/\mu $, $g'(s) = ln(1+ \mu s).$ Thus $g'(0) =0$, $g(s)>0$ if $s\not=0$. Now direct calculations show that
 $$
 g(s)\le \ln(1+\mu s)\,s\qquad\mbox{if }\; s\ge 0,
 $$
 and $g(s)\le |s|$ if $s\le 0$. Hence (i), (ii) and (iii) hold. By the
definition of $g$, (iv) clearly holds. Also $H(0) = 0$ and we get,
for $s \geq 0$,
$$H'(s) = \frac{1}{2}[g'(s)s - g(s)] = \frac{1}{2}[s - \frac{1}{\mu}ln (1 + \mu
s)].$$ Thus $ \displaystyle H''(s) = \frac{\mu s}{2(1+ \mu s)}
\geq 0$ for $s \geq 0$. From the convexity of $H$, we deduce
that, if $0 < s \leq t$,
$$H(s) \leq \frac{s}{t}H(t) + \left(1-\frac{s}{t}\right)H(0) = \frac{s}{t}H(t),$$ which
proves (v). Finally, we trivially check that
$$ H(s) = - G(- \frac{1}{\mu}) - \frac{1}{2\mu^2}$$
is constant for $s \leq - 1/\mu$, which implies (vi). The lemma is
proved.
\end{proof}

The next lemma concerns the geometrical structure
of $I$. We are going to denote with
$B(0,\rho)$  the ball in $X$  with center $0$ and radius $\rho$.

\begin{lemma}\label{mp}
Assume  $(H2)$. There exist constants
$\alpha = \alpha (N,|\Omega|,\mu) >0$, $\beta
>0$ and $\rho >0$ such that if $0 < \|c_0\|_{p} \leq \alpha$ then
\begin{itemize}
\item[(i)] $I(v) \geq \beta \, $  for $\, \|u\| =
\rho.$\smallskip
\item[(ii)] $\inf_{v \in B(0, \rho)}I(v) \le 0$, and\ \ $\inf_{v \in B(0, \rho)} I(v) < 0$ if $f\not\equiv0$. \smallskip
\item[(iii)] There exists  $v_0 \in X$ such that $\|v_0\|> \rho$ and
$I(v_0) \leq 0$.
\end{itemize}
\end{lemma}

\begin{proof}
 Let $r>1$, close to 1, satisfy $(r+1)p^\prime < \frac{2N}{N-2}.$
We can choose such $r$ since $p > \frac{N}{2}$. By Lemma~\ref{prop-g1}
we have
\begin{equation}\label{2.116}
|G(s)| \leq  C|s|^{r+1}, \quad
\mbox{ for all } s \in \R.
\end{equation}
Using~\eqref{2.116}, we get, for any $v \in X$,
\begin{equation}\label{2.1161}
       \int_{\Omega}c_0(x) G(v) dx \leq
C\|c_0\|_p \|v\|_{(r+1)p^{'}}^{r+1}\leq C
\|c_0\|_p \|v\|^{r+1},
  \end{equation}
where  we used the H\"older and Sobolev inequalities. Also
\begin{eqnarray*}
\int_{\Omega}f(x) v(x) dx \leq \|f\|_{\frac{N}{2}}\|v\|_{\frac{N}{N-2}} &\leq& D
(\|f^+\|_{\frac{N}{2}}+\|f^-\|_{\frac{N}{2}})\|v\|\\ &\leq& (D/\mu) (C_N+\mu\|f^-\|_{\frac{N}{2}})\|v\|, \end{eqnarray*} for some
$D=D(N,|\Omega|)>0$,  by the hypotheses of Theorem \ref{theorem1}.  We then get, for any $v \in X$, because
of~\eqref{smallc},
\begin{equation}\label{00}
I(v) \geq \frac{1}{2}\|v\|^2 - (D/\mu ) (C_N+\mu\|f^-\|_{\frac{N}{2}})\|v\|- C \|c_0\|_{p} \|v\|^{r+1}.
 \end{equation}  We fix first  $
\rho
>0$ sufficiently large so that if $\|v\| =\rho$
 $$\frac{1}{2}\|v\|^2 -
(D/\mu) (C_N+\mu\|f^-\|_{\frac{N}{2}})\|v\|\ge \frac{1}{4}\rho,$$ and then  $\|c_0\|_{p}$ small enough
to ensure that $I(v) \geq \frac{1}{8}\rho$, for any $v \in X$ with
$\|v\|=\rho$. This proves (i).

Next, note that $I(0)=0$, so $\inf_{v \in B(0, \rho)}I(v) \le 0$.
If $f\not\equiv0$,  take a function  $v \in C_0^{\infty}(\Omega)$,
such that $\int_{\Omega}f(x)v dx >0$ and consider the map $t \to
I(tv)$ for $t>0$. We have
\begin{align}\label{2.118}
       I(tv)  & = \frac{t^2}{2}\|v\|^2 - \int_{\Omega}c_0(x) G(tv)\thinspace
       dx - t \int_{\Omega}f(x)v \thinspace dx \\
       \nonumber
& = t^2 \left[ \frac{1}{2}\|v\|^2 - \int_{\Omega}c_0(x)
\frac{G(tv)}{t^2v^2}v^2  dx - \frac{1}{t}\int_{\Omega}f(x) v
\thinspace dx \right].
   \end{align}
By Lemma~\ref{prop-g1} we have $G(s)/s^2 \to 0$ as $s \to 0$, thus
$$\int_{\Omega}c_0(x) \frac{G(tv)}{t^2v^2}|v|^2 dx \to 0$$
as $t \to 0$, since $v \in C_0^{\infty}(\Omega)$. Then~\eqref{2.118} implies $I(tv) <0$ for
$t
>0$ small enough. This proves (ii).

Finally, to prove (iii) we consider again the map $t \to I(tv)$, $t>0$, and take $v \in C_0^{\infty}(\Omega)$ with  $v\geq 0$, $c_0v
\not\equiv 0$. Then since by Lemma~\ref{prop-g1} $G(s)/s^2 \to +
\infty$ as $s \to + \infty$, we now have
$$ \int_{\Omega}c_0(x)\frac{G(tv)}{t^2v^2}|v|^2 dx   \to +
\infty,$$ so $I(tv) \to - \infty$ as $t \to + \infty$. This of course implies (iii).
\end{proof}

In view of Lemma~\ref{mp} it can be expected that for $\|c_0\|_p$
sufficiently small $I$ has two critical points, one of which is a
local minimum, while the other is of saddle type.

\begin{lemma}\label{local-minima}
Assume  that  $\|c_0\|_p$ is sufficiently small to ensure that
$(H2)$ and Lemma~\ref{mp} hold. Then the functional $I$
possesses a critical point  $v\in B(0,\rho)$, with $I(v)\le 0$,
which is a local minimum of $I$.
\end{lemma}

\begin{proof}
By Lemma~\ref{mp} (i) and (ii) there are $\rho,\beta >0$ such that
$$ m := \inf_{v \in B(0, \rho)}I(v) \le 0 \quad \mbox{ and } \quad I(v) \ge \beta> 0
\quad \mbox{ if } \quad \|v\|= \rho.$$ Let $(v_n) \subset B(0,
\rho) \subset X$ be a sequence such that $I(v_n) \to m.$ Since
$(v_n) \subset X$ is bounded we have, up to a subsequence, $v_n
\rightharpoonup v$ weakly in $X$, for some $v \in X$. Now, by
standard properties of the weak convergence and since $f\in
L^{N/2}(\Omega)\subset X^{-1}$,
$$\|v\|^2 \leq \liminf_{n \to \infty}\|v_n\|^2 \quad \mbox{ and }
\quad \int_{\Omega}f(x)v_n \thinspace dx \to \int_{\Omega}f(x) v
\thinspace dx$$  as $n\to\infty$.Also, since $v_n \to v$ in $L^q(\Omega)$ for $1
\leq q < \frac{2N}{N-2}$ and  $c_0 \in L^p(\Omega)$ we readily
obtain, using~\eqref{2.116}, that
$$\int_{\Omega}c_0(x) G(v_n) dx \to \int_{\Omega}c_0(x) G(v) dx\quad \mbox{ as }n\to\infty.$$
We deduce that $v\in B(0,\rho)$ and
$$I(v) \leq \liminf_{n \to \infty} I(v_n) = m=\inf_{v \in B(0, \rho)}I(v).$$ Thus $v $ is a local minimum of $I$ and, by standard
arguments, a critical point of $I$.
\end{proof}

Now we define the mountain pass level
$$\hat c = \inf_{g \in \Gamma} \max_{t \in [0,1]} I(g(t))$$
where
$$\Gamma = \{ g \in C([0,1], X): g(0) =0, g(1) = v_0\},$$
with $v_0 \in X$ given by Lemma~\ref{mp} (iii). We shall prove
that $I$ possesses a critical point at the mountain pass level,
that is, there exists $v \in X$ such that $I(v) =\hat c$ and $I'(v)=0$. Since
$\hat c>0$ (by Lemma~\ref{mp} (i)), this critical point must be different from the local
minimum given by Lemma~\ref{local-minima}.
\medskip

It is a standard fact that any $C^1$-functional having a mountain pass geometry
admits a Cerami sequence at the mountain pass level (see for instance \cite{Ce,Ek}). In other words, there exists a sequence  $(v_n) \subset X$
such that
$$I(v_n) \to \hat c \qquad \mbox{ and } \qquad (1+ \|v_n\|) I'(v_n) \to 0.$$
If we manage to show that $(v_n) \subset X$ admits a strongly  convergent
subsequence, its limit is the desired critical point. A first
essential step in the proof of this fact is showing that $(v_n)$ is bounded.

\section{Boundedness of the Cerami sequences}\label{cerami}

The following lemma is the key point in the proof of Theorem \ref{theorem1}.

\begin{lemma}\label{bound-cerami} Assume  that  $\|c_0\|_p$ is
sufficiently small to ensure $(H2)$ and Lemma~\ref{mp} hold. Then the Cerami
sequences for $I$ at any level $d \in \R^+$ are bounded.
\end{lemma}

\begin{proof}  Let $(v_n)\subset X$ be a Cerami sequence for $I$ at a level $d \in \R^+$.
Assume for contradiction that $\|v_n\|
\rightarrow \infty$ and  set
$$ w_n = \frac{v_n}{\|v_n\|}.$$
Since $(w_n) \subset X$ is bounded we have  $w_n \rightharpoonup
w$ weakly in $X$ and $w_n \rightarrow w$ strongly in $L^q(\Omega),$ for $1 \leq q <
\frac{2N}{N-2}$ (up to a subsequence). We write $w = w^+ - w^-$. We shall distinguish the two cases $c_0w^+ \equiv 0$ and $c_0w^+ \not\equiv 0$, and prove they are both
impossible.

First we assume that $c_0w^+ =0$, and define the
 sequence $(z_n) \subset X$ by $z_n = t_n v_n$ with
$t_n \in [0,1]$ satisfying
\begin{eqnarray}\label{ex0}
I(z_n) = \max_{t \in [0,1]} I(tv_n)
\end{eqnarray}
(if  $t_n$ defined by (\ref{ex0}) is not unique we choose its
smallest possible value). Let us show that
\begin{eqnarray}\label{1000}
\lim_{n \rightarrow \infty}I(z_n) =  + \infty.
\end{eqnarray}
Seeking a contradiction we assume that for some $M < \infty$
\begin{eqnarray}\label{ex2}
\liminf_{n \rightarrow \infty}I(z_n) \leq  M,
\end{eqnarray}
and we define $( k_n) \subset
X$ by
\[ k_n  = \frac{\sqrt{4M}}{\|v_n\|}\, v_n=\sqrt{4M}w_n .\]
Then $k_n \rightharpoonup k:= \sqrt{4M} \,w$ weakly in $X$ and $k_n \to k$ strongly in
$L^q(\Omega)$ for any $1 \leq q < \frac{2N}{N-2}$. Thus, as in the proof of
Lemma~\ref{local-minima}, we have
\begin{equation}\label{1001}
\int_{\Omega}c_0(x) G(k_n) \thinspace  dx \to \int_{\Omega}c_0(x)
G(k) \thinspace dx.
\end{equation}
Now, recall that $G(s) \leq 0$ for $s \leq 0$, see Lemma~\ref{prop-g1}. Since we have assumed $c_0(x) =0$ if $k(x)>0$, we obtain
\begin{equation}\label{1002}
\int_{\Omega}c_0(x) G(k) \thinspace dx \leq 0.
\end{equation}
Also, since $f \in L^{N/2}(\Omega) \subset X^{-1}$
\begin{equation}\label{1003}
        \left| \int_{\Omega}f(x)k_n \thinspace dx \right|   \leq \sqrt{4M} \, \|f\|_{X^{-1}} \|w_n\| \leq  \sqrt{4M} \,
\|f\|_{X^{-1}}.
   \end{equation}
Combining~\eqref{1001},~\eqref{1002} and~\eqref{1003} it follows
that
\begin{align}\label{ex322}
I(k_n) & = 2 M -  \int_{\Omega}c_0(x) G(k_n) \thinspace dx- \int_{\Omega}f(x) k_n \thinspace dx \\
\nonumber & \geq 2 M - \sqrt{4M} \, \|f\|_{X^{-1}} + o(1).
\end{align}
Thus, taking $M>0$ larger if necessary, we can assume that
\begin{equation}\label{ex32}
I(k_n) \geq (3/2) M
\end{equation}  for all sufficiently
large $n \in \N$. Since $k_n$ and $z_n$ lay on the same ray in $X$ for all $n \in \N$,  we see by the definition of $z_n$ that
(\ref{ex32}) contradicts (\ref{ex2}) (note $\sqrt{4M}/\|v_n\|< 1$ since $\|v_n\|\to\infty$). Thus (\ref{1000}) holds.

We remark that $I(v_n)\to d$ and $I(z_n)\to\infty$ imply that $t_n\in (0,1)$. Hence  by the definition of $z_n$ we have  that $<\!I^\prime (z_n),z_n\!>\, = 0$, for all
$n \in \N$. Thus, with $H$ defined as in Lemma \ref{prop-g1},
\begin{align}\label{ex3}
        I(z_n) & = I(z_n) - \frac{1}{2} <\!I^\prime (z_n),z_n\!>  \\ \nonumber
&  =  \int_{\Omega} c_0(x)H(z_n)\thinspace dx -
\frac{1}{2}\int_{\Omega} f(x) z_n \thinspace dx.
   \end{align}
Combining (\ref{1000}) and (\ref{ex3}) we see that
\begin{equation}\label{ex333}
\frac{1}{2}\int_{\Omega}f(x) z_n \thinspace dx = - M(n) +
\int_{\Omega} c_0(x) H(z_n) \thinspace dx
\end{equation}
where $M(n)$ is a quantity such that $M(n) \to + \infty$ as $n\to\infty$. In order to show that $c_0w^+ =0$
does not occur we next prove that~\eqref{ex333}  is
impossible.

Observe that, for $n \in \N$ large
enough,
\begin{align}\label{ex34}
    d+1\ge    I(v_n) & = I(v_n) - \frac{1}{2} <\!I^\prime (v_n),v_n\!> +o(1) \\ \nonumber
&  =  \int_{\Omega} c_0(x)H(v_n)\thinspace dx -
\frac{1}{2}\int_{\Omega} f(x) v_n \thinspace dx + o(1)
   \end{align}
   (note that $<\!I^\prime (v_n),v_n\!>\, \to 0$, since $(v_n) $ is a Cerami sequence).
Thus, for some $D >0$,$$ \int_{\Omega} c_0(x) H(v_n)\thinspace dx
\leq D + \frac{1}{2}\int_{\Omega}f(x) v_n \thinspace dx = D +
\frac{1}{2t_n}\int_{\Omega} f(x) z_n \thinspace dx $$ or
equivalently, using~\eqref{ex333}
\begin{eqnarray}\label{ex37}
 \int_{\Omega} c_0(x) H(v_n) \thinspace dx \leq D -
\frac{M(n)}{t_n}+ \frac{1} {t_n}\int_{\Omega}c_0(x) H(z_n)
\thinspace dx.
\end{eqnarray}
Now we decompose $\Omega$ into $\Omega = \Omega_n^+ \cup \Omega_n^-$ with
$$ \Omega^+_n = \{ x \in \Omega : z_n(x)  \geq 0 \}\quad  \mbox{ and } \quad \Omega
^-_n = \Omega \backslash \Omega^+_n.$$ On $\Omega_n^+$ we have, by
Lemma~\ref{prop-g1} (v) and  $c_0 \geq 0$, that
$$ \int_{\Omega_n^+}c_0(x) H(z_n)\thinspace dx \leq t_n \int_{\Omega_n^+}c_0(x)
H(v_n)\thinspace dx.$$ On $\Omega_n^-$ we have, by
Lemma~\ref{prop-g1} (vi) and $|\Omega| < \infty$, that for some $D
>0$  $$ \int_{\Omega_n^-}c_0(x) H(z_n)\thinspace dx \leq D.$$ Then it follows
from (\ref{ex37})  that
$$
  \int_{\Omega_n^-}c_0(x) H(v_n) \thinspace dx
        \leq
D - \frac{M(n)}{t_n} + \frac{D}{t_n}.
$$
Letting $n \to \infty $ and using  $t_n\in [0,1]$ we see that
$$ \int_{\Omega_n^-}c_0(x) H(v_n)\thinspace dx \to - \infty$$
which is impossible since, by Lemma \ref{prop-g1} (vi), $H$ is bounded on
$\R^-$ and $|\Omega| < \infty$. At this point we have shown that $c_0w^+ = 0 $
is impossible.
\medskip

We now assume that $c_0w^+ \neq 0$ and we show that this property also leads to a
contradiction.  Since $(v_n) \subset X$ is a Cerami sequence we have $<\!I^\prime (v_n),v_n\!>\, \to 0$. Thus
$$ \|v_n\|^2 -  \int_{\Omega}c_0(x) g(v_n) v_n \thinspace dx -
\int_{\Omega} f(x) v_n \thinspace dx \to 0.$$ Dividing by
$\|v_n\|^2$ we get
$$ \|w_n\|^2 -  \int_{\Omega}c_0(x)
\frac{g(v_n)}{\|v_n\|}w_n \thinspace dx \to 0,$$ and since
$\|w_n\| =1$ we have
\begin{eqnarray}\label{ex38}
 \int_{\Omega}c_0(x) \frac{g(v_n)}{\|v_n\|}w_n \thinspace dx =  \int_{\Omega}c_0(x) \frac{g(v_n)}{v_n}w^2_n \thinspace dx \to 1.
\end{eqnarray}
Let $$\Omega^+ = \{x \in \Omega: c_0(x)w(x) > 0 \}\not=\emptyset.$$ We also define
$$ \Omega_n^+ = \{ x \in \Omega : v_n(x)  \geq 0 \} \quad \mbox{ and } \quad \Omega
^-_n = \Omega \backslash \Omega^+_n.$$ Now since $g(s)/s \to + \infty$ as $s
\to + \infty$ and $w_n \to w>0 $ a.e. on $\Omega^+$ it follows that
$$  c_0\frac{g(v_n)}{v_n}w_n^2 \to +
\infty \quad \mbox{a.e. on } \Omega^+. $$  Thus, taking into
account  that $|\Omega^+|
>0$, we deduce that
\begin{equation}\label{ex39}
 \lim_{n \to  \infty} \int_{\Omega^+}c_0(x)
\frac{g(v_n)}{v_n}w^2_n \thinspace dx \to + \infty.
\end{equation}
On the other hand we have
\begin{align}\label{ex40}
        \int_{\Omega^+} c_0(x) \frac{g(v_n)}{v_n}w_n^2 \thinspace dx& = \int_{\Omega} c_0(x) \frac{g(v_n)}{v_n}w_n^2 \thinspace dx  \\ \nonumber
&  - \int_{(\Omega\backslash \Omega^+)\cap \Omega_n^+}c_0(x)
\frac{g(v_n)}{v_n}w_n^2 \thinspace dx \\ \nonumber & -
\int_{(\Omega\backslash \Omega^+)\cap \Omega_n^-}c_0(x)
\frac{g(v_n)}{v_n}w_n^2 \thinspace dx.
   \end{align}
   But, for all $n \in \N$, since $g$ is non negative,
\begin{equation}\label{ex41}
\int_{(\Omega\backslash \Omega^+)\cap \Omega_n^+}
c_0(x)\frac{g(v_n)}{v_n}w_n^2\thinspace dx \geq 0.
\end{equation}
Also, since $g(s)/s$ is bounded for $s \leq 0$ we have, for some $D >0$,
\begin{align}\label{ex42}
\left| \int_{(\Omega\backslash \Omega^+)\cap \Omega_n^-}c_0(x)
\frac{g(v_n)}{v_n}w_n^2 \thinspace dx \right|  & \leq D \int_{\Omega}c_0(x) w_n^2 \thinspace dx \\
\nonumber & \leq D \|c_0\|_{\frac{N}{2}} \|w_n\|^2 \leq D
\|c_0\|_{\frac{N}{2}}.
\end{align}

Now combining~\eqref{ex38}-\eqref{ex42} we get a contradiction. This shows that $c_0w^+ \neq 0$ is impossible and ends the proof of the lemma.
\end{proof}

\begin{lemma}\label{convergence} Under the hypotheses of
Lemma~\ref{bound-cerami} any Cerami sequence for $I$ at a level $d
\in \R^+$ admits a strongly convergent subsequence.
\end{lemma}

\begin{proof}
Let $(v_n) \subset X$ be a Cerami sequence for $I$ at a level $d
\in \R^+$. Since by Lemma~\ref{bound-cerami} this sequence is
bounded, by passing to a subsequence  we can
assume that $v_n \rightharpoonup v$ weakly in $X$ and $v_n \to v$ strongly  in $L^q(\Omega)$, for each $1 \leq q < \frac{2N}{N-2}$. The condition
$I'(v_n) \to 0$ in $X^{-1}$ means precisely that
$$ - \Delta v_n - [c_0(x) + \mu f(x)] v_n - c_0(x) g(v_n) - f(x)
\to 0 \quad \mbox{ in } X^{-1}.$$ Because $v_n \to v$ in
$L^q(\Omega)$, for $1 \leq q < \frac{2N}{N-2}$ and $c_0 \in
L^p(\Omega)$ for some $p > \frac{N}{2}$ we readily have that $c_0(x)
g(v_n) \to c_0(x) g(v)$ in $X^{-1}$. Thus
\begin{equation}\label{001}
- \Delta v_n - [c_0(x) + \mu f(x)] v_n \to c_0(x) g(v) + f(x)
\quad \mbox{ in } X^{-1}.
\end{equation}
Now let $L : X \to X^{-1}$ be defined by
$$(Lu)v = \int_{\Omega}\nabla u \nabla v - [c_0(x) + \mu f(x)] uv \thinspace
dx.$$ The operator $L$ is invertible by \eqref{smallc}, so  we can deduce
from~\eqref{001} that $v_n \to L^{-1}[c_0(x)g(v) + f(x)]$ in $X$.
Consequently, by the uniqueness of the limit, $v_n \to v$ in $X$.
\end{proof}

\section{Proofs of the main theorems}\label{nonconstant}

With the results from the previous section at hand, we are ready
to prove  Theorem \ref{theorem1}.  We assume that $\ov{c}>0$ is chosen
sufficiently small to ensure that the conclusions of
Lemmas~\ref{positivity}--\ref{convergence}  hold.
\medskip

 \noindent {\it Proof of Theorem~\ref{theorem1}.} Let first   $\mu>0$.
By Lemma~\ref{local-minima} we have the existence of a first
critical point which is a local minimum of $I$, whereas by
Lemmas~\ref{bound-cerami} and~\ref{convergence} we obtain a second
critical point at the mountain pass level $\hat c >0$. So we obtain two different solutions of ~\eqref{2.11} in $X$. By Lemma~\ref{negative-bound} and Lemma \ref{dual1} they give  two different solutions of~\eqref{problem0}. These solutions are bounded, as a consequence of Lemma \ref{regularity} below.

Next, if $\mu<0$ we replace $u$ by $-u$, which is equivalent to replacing $\mu$ by $-\mu$ and $f$ by $-f$.  Theorem~\ref{theorem1} is proved. \qed
\medskip

Now consider the equation
\begin{equation} \label{moregen}
- \mathrm{div}(A(x) \nabla u) =   \mu <\!\!A(x)\nabla u,\nabla u\!\!>   +c_0(x)u + f(x),
\end{equation}
and assume $\Lambda I\ge A(x)\ge \lambda I$, where $\Lambda\ge \lambda>0$. We
have just proved Theorem \ref{theorem1} for \eqref{moregen} with $A(x)=I$.

It is trivial to check that   the change of unknown $v =
\frac{1}{\mu} (e^{\mu u}-1)$  transforms \eqref{moregen} into
\begin{equation} \label{moregen1}
- \mathrm{div}(A(x) \nabla v) - \left[c_0(x) + \mu f(x)\right] v = {c_0(x)} g(v) +
f(x).
\end{equation}
 This equation is variational and can be treated exactly
like~\eqref{2.11}. Repeating the arguments from the previous sections we are led to the following result.
\begin{theorem} \label{theorem11} Assume that $$c_0\gneqq0\;\mbox{ in } \Omega\quad\mbox{  and }\quad\mu\not=0.$$
If
$$
\|[\mu f]^+\|_{L^{\frac{N}{2}}(\Omega)}
<{\lambda C_N}
$$
and
$$
\max\{\|c_0\|_{L^p(\Omega)}\,,\,\-[\mu f]^-\|_{L^p(\Omega)}\}<
\ov{c}\,,
$$
where $\ov{c}>0$ depends only  on $N$, $p$, $\lambda$, $\Lambda$, $|\Omega|$,
$|\mu|$, $\|[\mu f]^+\|_{L^p(\Omega)}$, then \eqref{moregen}
admits at least  $\mathrm{two}$ $\mathrm{bounded}$ solutions.
\end{theorem}

The boundedness of the solutions obtained in this theorem (which contains Theorem \ref{theorem1} as a particular case) is a consequence of the following lemma.
\begin{lemma}\label{regularity}
Assume that $\Lambda I \ge A(x)\ge \lambda I$ for some $\Lambda\ge \lambda>0$, $\mu \in L^{\infty}(\Omega)$, and that $c_0$ and $f$
belong to $L^p(\Omega)$, for some $p > \frac{N}{2}$. Then any
solution $v\in X$ of~\eqref{moregen1} belongs to
$L^\infty(\Omega)$.
\end{lemma}
\begin{proof}
Let $v \in X$ be a solution of~\eqref{moregen1}, which we recast as
$$ - \mathrm{div}(A(x) \nabla v) = [c_0(x) + \mu f(x) + c_0(x) \frac{g(v)}{v}] v +  f(x). $$
By our assumptions $c_0$, $f$, and $\mu f$ belong to
$L^{p}(\Omega)$, for some $p>N/2$. We will be in position to apply Lemma \ref{trudi} provided we show that the term $c_0(x) \frac{g(v)}{v}$ has the same property.
 This is indeed the case
because of the slow growth of $g(s)/s$ as $|s| \to \infty$ (recall Lemma~\ref{prop-g1}). Specifically, for any $r\in (0,1)$ there
exists a $D >0$ such that
$$\left|\frac{g(s)}{s}\right| \leq D |s|^r, \quad \mbox{ for any } s \in \R.$$
Thus, since $c_0 \in L^p(\Omega)$ for some $p > \frac{N}{2}$, and $v$ is in some Lebesgue space ($v\in L^{2N/(N-2)}(\Omega)$), by taking
$r >0$ sufficiently small ($r <
\frac{4p-2N}{p(N-2)}$) and by using the H\"older inequality we see  that  $c_0
g(v)/v \in L^{p_1}(\Omega)$, for some $p_1\in (N/2,p)$.

So $v$ is bounded, by Lemma \ref{trudi}.
\end{proof}

We are now ready to prove Theorem \ref{theoexist}. The idea is to use Theorem~\ref{theorem11} in order to obtain a supersolution and a subsolution to \eqref{Generalf} which can be proved to be ordered. Then we can obtain the existence of
one solution to \eqref{Generalf} by  appealing to a theorem which states the existence of a
solution between ordered sub- and super-solutions. Such results
abound in the theory of elliptic PDE, see for instance
\cite{AmCr}, \cite{DeHe}, and the references in these works. We
are going to use Theorem 3.1 from \cite{BoMuPu4}, which is
particularly adapted to our setting.

We recall, see Definition 3.1 of \cite{BoMuPu4}, that a function $\underline{w}
\in W^{1,2}(\Omega)\cap L^\infty(\Omega)$ is a subsolution of~\eqref{Generalf} if
$$- \mathrm{div}(A(x) \nabla  \underline{w}) \leq H(x, \underline{w}(x), \nabla \underline{w} (x))\quad  \mbox{ in }  \Omega
\quad \mbox { and } \quad \underline{w} \leq 0 \, \mbox{ on } \, \partial \Omega.$$ Respectively, a
function $\ov{w} \in W^{1,2}(\Omega)\cap L^\infty(\Omega)$ is a supersolution of~\eqref{Generalf} if
$$ - \mathrm{div}(A(x) \nabla  \ov{w})\geq H(x, \ov{w}(x), \nabla \ov{w} (x))\quad  \mbox{ in } \, \Omega
\quad \mbox { and } \quad \ov{w} \geq 0 \, \mbox{ on } \, \partial
\Omega.$$ The function $H$ obviously satisfies the hypothesis
(1.5) from \cite{BoMuPu4}, since
$$
|H(x,s,\xi)|\le (1+\mu +|s|)(|c_0(x)|+|f(x)|+|\xi|^2).
$$

\noindent {\it Proof of Theorem~\ref{theoexist}.}
Observe that any solution $\ov{u}\in X$ of the equation
  \begin{equation} \label{more}
- \mathrm{div}(A(x) \nabla u) =   \frac{\mu}{\lambda} <\!\!A(x)\nabla u,\nabla u\!\!>   +c_0(x)u + f(x),
\end{equation}
given by applying Theorem \ref{theorem11}  to \eqref{more}
 in the particular case  $\mu>0$ and $f\ge 0$, is such that $\ov{u}$ is a  {\it supersolution} to the original equation \eqref{Generalf}, thanks to (H1). In addition $\ov{u}\ge 0$, by the maximum principle (Lemma \ref{positivity}) which can be applied to \eqref{moregen1}.  Similarly, it is easily checked that $\underline{u}= -\ov{u}$ is a subsolution to \eqref{Generalf}, and of course $\underline{u}\le 0\le \overline{u}$, so $\underline{u}$, $\ov{u}$  are ordered. These functions are bounded, by Lemma~\ref{regularity}. Thus
 Theorem 3.1 of \cite{BoMuPu4} implies  Theorem~\ref{theoexist}.
\qed

\section{Final Remarks}\label{final}

The
hypotheses we made on \eqref{0.2} in order to prove Theorem
\ref{theorem1} can be generalized in various ways. For instance,
if $f \geq 0$ we see that any solution
$v$ of~\eqref{2.11} satisfies $v \geq 0$ on $\Omega$, provided
$\|c_0+\mu f\|_{N/2}<C_N$ (by Lemma \ref{positivity}). Then  we do not need any more
Lemma~\ref{negative-bound} and, inspecting the proofs of the
remaining lemmas, one can see that requiring that $c_0$ belongs to
$L^p(\Omega)$ for some $ p> \frac{N}{2}$ and that $f \in
L^{\frac{N}{2}}(\Omega)$ suffices to get the conclusion of
Theorem~\ref{theorem1}. In this case the solutions
of~\eqref{2.11} and thus of~\eqref{problem0} which we obtain are not necessarily
bounded. One may in general ask whether it is possible to consider
coefficients $c_0$ and $f$ which are less regular, thus obtaining
solutions with lower regularity, like for instance in
\cite{GrMuPo}.

Let us also make some remarks on the importance of the change of variables
$u=\frac{1}{\mu}\ln(1+\mu v)$ which we used. If the operators div$(a)$ and $B$
in \eqref{0.2} can be appropriately bounded above and below by quantities such
that this change can be made in the corresponding ``extremal" equations,
leading to new equations for which our critical point method can be applied,
then we obtain a subsolution and a supersolution for the initial problem, and hence
a solution to this problem. This approach is in some sense alternative, as well as complementary,  to the one used in many previous papers on coercive problems with natural growth. In these papers the idea was to mimic the change of variables in the initial problem, by testing the weak formulation of \eqref{0.2} with suitably chosen functions, which somehow take account of the change of unknown (see for instance Remark 2.10 in \cite{FeMu2} for more details).

We stress  that, in contrast with the existence result in Theorem \ref{theoexist} which is very general with respect to the structure of the equation, our multiplicity result, Theorem \ref{theorem11}, depends on a strict link between the second order term and the gradient term. In other words, to obtain multiplicity we do need to
be able to make the change of variables in the initial equation. It is certainly a very interesting open problem whether a multiplicity result can be proved for more general non-coercive problems with natural growth. It can be expected that topological methods, in particular index theory, should permit us to deduce multiplicity of solutions for equations which do not have an equivalent formulation in terms of critical points of a functional. For instance, we state

\begin{open} Under appropriate smallness condition on $c_0$ and $f$ is it true that the equation
$$
-\Delta u = c_0(x)u + \mu(x) |\nabla u|^2 + f(x), \qquad  u \in
X,
$$
has at least two bounded solutions, provided $0<\mu_1\le \mu(x)\le \mu_2$ and $\mu$ is not constant ? \end{open}

\noindent{\it Acknowledgement}. We thank M.-F. Bidaut-V\'eron for a number of very useful remarks, which improved our presentation.
\bigskip


\end{document}